\documentclass[12pt, reqno, twoside, letterpaper]{amsart}

\usepackage{paperstyle}
\usepackage{graphicx}

\usepackage{mathtools}
\usepackage{todonotes}
\usepackage[norefs,nocites]{refcheck}
\usepackage{amsmath,amsthm,amscd,amssymb}

\newcommand{\be}{\begin{equation}}
\newcommand{\ee}{\end{equation}}
\newcommand{\dalign}[1]{\[\begin{aligned} #1 \end{aligned}\]}

\title[Comparing zeros of distinct Dirichlet $L$-functions]
{Comparing zeros of\\ distinct Dirichlet $L$-functions}
      
\author[W.~Banks]{William Banks}

\address{Department of Mathematics, 
         University of Missouri, 
         Columbia MO 65211, USA.}

\email{bankswd@missouri.edu}
        
\date{\today}

\begin{document}

\begin{abstract}
For any $\theta>\frac13$, we show that there are constants $c_1,c_2>0$ that 
depend only on $\theta$ for which the following property holds.
If $\chi_1,\chi_2$ are two distinct primitive
Dirichlet characters modulo $q$, and $T\ge c_1q^\theta$,
then $L(s,\chi_1)$ and $L(s,\chi_2)$ do not have
the same zeros in the region
$$
\big\{s=\sigma+it\in\C:0<\sigma<1,~T<t<T+c_2q^\theta\log T\big\}.
$$
For cubefree moduli $q$, the same result holds for any $\theta>\frac14$.
\end{abstract}

\thanks{
MSC Primary: 11M06; Secondary: 11M26.}


\maketitle


\section{Introduction}

For a given Dirichlet character $\chi$ modulo $q$, the
$L$-function $L(s,\chi)$ is defined in the half-plane $\sigma>1$ via the
equivalent definitions
$$
L(s,\chi)\defeq\sum_{n\ge 1} \chi(n)n^{-s}
=\prod_{p\text{~prime}}(1-\chi(p)p^{-s})^{-1}.
$$
If $\chi$ is nonprincipal, then
$L(s,\chi)$ extends to an entire function on the complex plane, and
if $\chi$ is primitive, then $L(s,\chi)$ satisfies
a simple functional equation relating its values at
$s$ and $1-s$. In this note, we prove the following:

\begin{theorem}\label{thm:main1}
For any $\theta>\frac{1}{3}$,
there are constants $c_1,c_2>0$ that depend only on~$\theta$ such that
the following property holds.
Let $\chi_1$ and $\chi_2$ be distinct primitive Dirichlet characters modulo $q$.
Then, for any $T\ge c_1 q^\theta$, the functions $L(s,\chi_1)$ and $L(s,\chi_2)$ 
have different zeros $($counted with multiplicity$\,)$ in the region
$$
\cR(q,T)\defeq\big\{s=\sigma+it\in\C:
0<\sigma<1,~T<t<T+c_2 q^\theta\log T\big\}.
$$
In other words, the function
\be\label{eq:LL}
\frac{L(s,\chi_1)}{L(s,\chi_2)}+\frac{L(s,\chi_2)}{L(s,\chi_1)}
\ee
has at least one pole in $\cR(q,T)$.
Moreover, if $q$ is cubefree, then the same result holds for
any $\theta>\frac14$.
\end{theorem}

Some related results have been found previously.
Fujii~\cite[Thm.~$1'$]{Fujii} showed that if $\chi_1$ and $\chi_2$ are
distinct primitive characters to the same modulus $q$,
then a positive proportion of the nontrivial zeros
of $L(s,\chi_1)L(s,\chi_2)$ are poles of the function \eqref{eq:LL}.
Moreover, it follows from \cite[\S3 Cor.\,$(i)$]{Fujii} that
$L(s,\chi_1)$ and $L(s,\chi_2)$ have at least $c_3\log T$
such zeros in the region $\cR(q,T)$ for a positive proportion of
$T\in(T_1,T_1+T_1^{1/2+\eps})$ if
$T_1\ge c_4$ and $H>T_1^{1/2+\eps}$, where $c_3,c_4>0$ are
constants that depend only $q$ and $\eps$;
this result also holds generally for distinct Dirichlet
\text{$L$-functions} associated to primitive Dirichlet characters whose
moduli may be different (see Fujii \cite[\S 4]{Fujii}). Thus, Fujii's method
produces more such zeros than Theorem~\ref{thm:main1} in short rectangles of the form
$\{0<\sigma<1,~T_1<t<T_1+H\}$, but the constants are ineffective
in $q$ and $\eps$. By contrast, Theorem~\ref{thm:main1} produces fewer zeros in $\cR(q,T)$
on average, but the holds for all  large $T$.

Using a variant of Hamburger's theorem,
Raghunathan~\cite[Thm.~3]{Ragun} showed that for any pair of
distinct primitive characters $\chi_1$ and $\chi_2$, both functions
$$
\frac{L(s,\chi_1)}{L(s,\chi_2)}\mand\frac{L(s,\chi_2)}{L(s,\chi_1)}
$$
have infinitely many non-integer poles.

Our proof of Theorem~\ref{thm:main1} is based on a generalization of
a striking formula due to Landau~\cite[Satz~1]{Landau},
which asserts that for any fixed $x>1$ one has
\be\label{eq:magic}
\ssum{\rho=\beta+i\gamma\\1<\gamma\le T}x^\rho
=-\frac{T}{2\pi}\Lambda(x)\ind{\Z}(x)+O(\log T),
\ee
where the sum runs over the multiset of nontrivial zeros $\rho=\beta+i\gamma$
of the Riemann zeta function $\zeta(s)$ with $1<\gamma\le T$,
$\Lambda$ is von Mangoldt function, and $\ind{\Z}$ is the indicator
function of the integers:
$$
\ind{\Z}(x)\defeq
\begin{cases}
1&\quad\hbox{if $x\in\Z$},\\
0&\quad\hbox{otherwise}.
\end{cases}
$$
Gonek~\cite[Thm.~1]{Gonek} gave a uniform version of Landau's
result \eqref{eq:magic} with an error term expressed in terms of both
variables $x$ and $T$. Some stronger
versions of the Gonek-Landau theorem are known; see
\cite{FordZahar,FujiiLand1,FujiiLand2}.
The following theorem is an extension of \cite[Thm.~1]{Gonek},
which is proved using Gonek's methods, covering the case of an $L$-function
attached to a primitive Dirichlet character.

\begin{theorem}\label{thm:main2}
Let $\chi$ be a primitive Dirichlet character modulo $q$.
Uniformly for $x\ge 2$ and $T_2>T_1>1$, we have
\be\label{eq:LandauDirichlet}
\ssum{\rho=\beta+i\gamma\\T_1<\gamma<T_2}x^\rho
=-\frac{T_2-T_1}{2\pi}\Lambda(x)\chi(x)\ind{\Z}(x)+O(E),
\ee
where the sum runs over the nontrivial zeros $\rho=\beta+i\gamma$
of $L(s,\chi)$ for which $T_1<\gamma<T_2$ $($each zero being summed
according to its multiplicity$)$, 
$$
E\defeq x\log x\,\log\log 2x
+x\log x\,\min\Big\{\frac{T_2}{x},\frac{1}{\langle x\rangle}\Big\}
+x\,\log\log 2x\,\log 2qT_2,
$$
and $\langle x\rangle$ is the distance from $x$ to the nearest
prime power other than $x$ itself.
The implied constant in \eqref{eq:LandauDirichlet} is absolute.
\end{theorem}

We remark that similar results have been achieved covering
all functions in the Selberg class (including the Dirichlet $L$-functions
considered here); see, for example, \cite{SSS,FordSoundZahar,MurtyPerelli,MurtyZahar}.
For our purposes, the error term provided
by Theorem~\ref{thm:main2} is convenient as the dependence on 
the quantity $\langle x\rangle$ is explicit.

Theorem~\ref{thm:main2} is proved in \S\ref{proof main 2} and
is used in the proof of Theorem~\ref{thm:main1} (see \S\ref{proof main 1}).

\section{Preliminaries}

\begin{lemma}\label{lem:G}
For $x\ge 2$, $T>1$, and $c\defeq 1+\tfrac{1}{\log x}$, we have
$$
\ssum{n\ge 2\\n\ne x}\frac{\Lambda(n)}{n^c}
\min\Big\{T,\frac{1}{|\log(x/n)|}\Big\}\ll\log x\,\log\log 2x
+\log x\,\min\Big\{\frac{T}{x},\frac{1}{\langle x\rangle}\Big\}.
$$
\end{lemma}

\begin{proof}
See Gonek~\cite[Lem.~2]{Gonek}.
\end{proof}

\begin{lemma}\label{lem:MV}
Let $\chi$ be a primitive character modulo $q\ge 1$. Then the estimate
$$
\frac{L'}{L}(s,\chi)=\ssum{\rho\\|\gamma-t|<1}\frac{1}{s-\rho}
+O\big(\log 2q|t|\big)
$$
holds uniformly for $-1\le\sigma\le 2$ and $|t|\ge 1$.
\end{lemma}

\begin{proof}
See Montgomery and Vaughan \cite[Lems.\ 12.1 and 12.6]{MontVau}.
\end{proof}

\begin{lemma}\label{lem:Burgess}
For any $\theta>\frac{1}{3}$,
there are constants $c_3,c_4>0$ that depend only on~$\theta$ such that
the following property holds. If $\chi$ is a nonprincipal character modulo $q\ge 3$,
then there exists a prime $p_0$ such that
\be\label{eq:claim1}
p_0\le c_3 q^\theta
\mand
\big|\chi(p_0)-1\big|\ge\frac{c_4}{(\log q)^2}.
\ee
If $q$ is cubefree, then the same result holds for
any $\theta>\frac14$.
\end{lemma}

\begin{proof}
It suffices to prove this when $q$ is sufficiently large (depending on $\theta$).

The well known results of Burgess~\cite{Burgess1,Burgess2,Burgess3}
state that the bound
\be\label{eq:burgess}
\sum_{N<n\le N+H}\chi(n)\mathop{\,\ll\,}\limits_{\eps,r}
H^{1-1/r}q^{(r+1)/(4r^2)+\eps}
\ee
holds for $N\in\R$, $H\ge 1$, $\eps>0$, and any
integer $r\ge 2$, provided that $q$ is cubefree or that $r\le 3$;
the implied constant depends only on $\eps$ and $r$.
Applying \eqref{eq:burgess}
with $N\defeq 0$, $H\defeq c_3 q^\theta$, $\eps\defeq\frac14(\theta-\frac13)$,
and $r\defeq 3$, we have
\be\label{eq:burgess-arbitrary}
\sum_{n\le c_3 q^\theta}\chi(n)\mathop{\,\ll\,}\limits_\theta
q^{\theta-\frac{1}{12}(\theta-\frac13)},
\ee
where the implied constant depends only on $\theta$.

On the other hand, assume there is no prime $p_0$ satisfying
\eqref{eq:claim1}. Then  $\chi(p_0)=1+O_\theta((\log q)^{-2})$ holds
for all primes $p_0\le c_3 q^\theta$, and thus
$\chi(n)=1+O_\theta((\log q)^{-1})$ holds for all natural numbers
$n\le c_3 q^\theta$ since each $n$ has at most $O_\theta(\log q)$ prime factors. Consequently,
$$
\sum_{n\le c_3 q^\theta}\chi(n)\mathop{\,\gg\,}\limits_\theta q^\theta;
$$
this is impossible in view of \eqref{eq:burgess-arbitrary}. Thus, the
assumption is erroneous, and the lemma follows for general $q$.

When $q$ is cubefree, we apply
\eqref{eq:burgess} with the choices $N\defeq 0$,
$H\defeq c_3 q^\theta$, and
$\eps\defeq\frac{1}{8r}(\theta-\frac14)-\frac{1}{4r^2}$,
where the integer $r$ is large enough so that $\eps>0$.
In this case, we have
\be\label{eq:burgess-cubefree}
\sum_{n\le c_3 q^\theta}\chi(n)\mathop{\,\ll\,}\limits_\theta
q^{\theta-\frac{7}{8r}(\theta-\frac14)},
\ee
and the argument can be completed as above.
\end{proof}

\section{Proof of Theorem~\ref{thm:main2}}
\label{proof main 2}

Our proof closely parallels that of Gonek~\cite[Thm.~1]{Gonek}.

We can assume without loss of generality that neither $T_1$ nor $T_2$
is the ordinate of a zero of $L(s,\chi)$.
Put $c\defeq 1+\tfrac{1}{\log x}$ and denote by
$\cC$ the oriented rectangle:
$$
c+iT_1
~~\longrightarrow~~c+iT_2
~~\longrightarrow~~1-c+iT_2
~~\longrightarrow~~1-c+iT_1
~~\longrightarrow~~c+iT_1.
$$
By Cauchy's theorem we have
\dalign{
\ssum{\rho=\beta+i\gamma\\T_1<\gamma<T_2}x^\rho
&=\frac{1}{2\pi i}\bigg(\int_{c+iT_1}^{c+iT_2}
+\int_{c+iT_2}^{1-c+iT_2}
+\int_{1-c+iT_2}^{1-c+iT_1}
+\int_{1-c+iT_1}^{c+iT_1}\bigg)
\frac{L'}{L}(s,\chi)\,x^s\,ds\\
&=I_1+I_2+I_3+I_4\quad\text{(say)}.
}
We estimate the four integrals separately.

In the half-plane $\sigma>1$ we have
\be\label{eq:L'LDirSer}
\frac{L'}{L}(s,\chi)=-\sum_{n\ge 2}\frac{\Lambda(n)\chi(n)}{n^s},
\ee
and therefore
\dalign{
I_1&=\frac{1}{2\pi}\int_{T_1}^{T_2}\frac{L'}{L}(c+it,\chi)\,x^{c+it}\,dt
=-\frac{x^c}{2\pi}\sum_{n\ge 2}\frac{\Lambda(n)\chi(n)}{n^c}
\int_{T_1}^{T_2}\Big(\frac{x}{n}\Big)^{it}\,dt\\
&=-\frac{T_2-T_1}{2\pi}\Lambda(x)\chi(x)\ind{\Z}(x)
+O\bigg(x^c\ssum{n\ge 2\\n\ne x}\frac{\Lambda(n)}{n^c}
\min\Big\{T_2,\frac{1}{|\log(x/n)|}\Big\}\bigg).
}
Applying Lemma~\ref{lem:G} and taking into account that
$x^c\ll x$, we see that
\be\label{eq:I1result}
I_1=-\frac{T_2-T_1}{2\pi}\Lambda(x)\chi(x)\ind{\Z}(x)
+O\bigg(\! x\log x\bigg(\!\log\log 2x
+\min\Big\{\frac{T_2}{x},\frac{1}{\langle x\rangle}\Big\}\bigg)\!\bigg).
\ee

Next, using Lemma~\ref{lem:MV} we have
$$
I_2=\frac{1}{2\pi i}\ssum{\rho\\|\gamma-T_2|<1}
\int_{c+iT_2}^{1-c+iT_2}\frac{x^s\,ds}{s-\rho}
+O\bigg(\log 2qT_2\,\int_{1-c}^cx^{\sigma}\,d\sigma\bigg),
$$
where the sum runs over the zeros $\rho=\beta+i\gamma$
of $L(s,\chi)$ for which $|\gamma-T_2|<1$, each zero being summed
according to its multiplicity. Since $x^c\ll x$, the error term
is $O(x\log 2qT_2/\log x)$, and arguing as in \cite[p.\,402]{Gonek}
we have for each zero $\rho$ in the sum:
$$
\int_{c+iT_2}^{1-c+iT_2}\frac{x^s\,ds}{s-\rho}
\ll x\,\log\log 2x.
$$
Since there are at most $O(\log 2qT_2)$ such zeros, we obtain that
\be\label{eq:I2result}
I_2\ll x\,\log\log 2x\,\log 2qT_2.
\ee
Similarly,
\be\label{eq:I4result}
I_4\ll x\,\log\log 2x\,\log 2qT_1.
\ee

Finally, we come to
\be\label{eq:I3remembered}
I_3=-\frac{1}{2\pi}\int_{T_1}^{T_2}\frac{L'}{L}(1-c+it,\chi)\,x^{1-c+it}\,dt.
\ee
We recall the functional equation
(cf.\ \cite[Eq.~(10.35)]{MontVau}):
\be\label{eq:funcl eqn}
\frac{L'}{L}(s,\chi)=-\frac{L'}{L}(1-s,\overline\chi)-\log\frac{q}{2\pi}
-\frac{\Gamma'}{\Gamma}(1-s)
+\frac{\pi}{2}\cot\frac{\pi}{2}(s+\kappa)
\ee
where
$$
\kappa=\kappa(\chi)\defeq\begin{cases}
0&\quad\text{if $\chi(-1)=1$},\\
1&\quad\text{if $\chi(-1)=-1$}.\\
\end{cases}
$$
Setting $s\defeq 1-c+it$ for some $t\in(T_1,T_2)$,
and using the estimate \cite[Eq.~(C.17)]{MontVau} for the digamma function $\Gamma'/\Gamma$,
we derive that
\dalign{
-\frac{L'}{L}(1-c+it,\chi)&=\frac{L'}{L}(c-it,\overline\chi)
+\log\frac{q}{2\pi}+\frac{\Gamma'}{\Gamma}(c-it)-
\tfrac{\pi}{2}\cot\tfrac{\pi}{2}(1-c+\theta+it)\\
&=\frac{L'}{L}(c-it,\overline\chi)
+\log\frac{q\,t}{2\pi}+O(t^{-1}),
}
where the implied constant is absolute. Inserting this estimate
into \eqref{eq:I3remembered}, and taking into account that
$x^{1-c}\ll 1$, it follows that
$$
I_3=
\frac{x^{1-c}}{2\pi}\int_{T_1}^{T_2}
\frac{L'}{L}(c-it,\overline\chi)\,x^{it}\,dt
+O\(\frac{\log 2qT_2}{\log x}+\log T_2\).
$$
Using \eqref{eq:L'LDirSer}, the first term is
\dalign{
-\frac{x^{1-c}}{2\pi}\sum_{n\ge 2}\frac{\Lambda(n)\chi(n)}{n^c}
\int_{T_1}^{T_2} (nx)^{it}\,dt
&\ll\sum_{n\ge 2}\frac{\Lambda(n)|\chi(n)|}{n^c\log nx}\\
&\ll-\frac{\zeta'}{\zeta}(c)\frac{1}{\log x}\ll\frac{1}{(c-1)\log x}=1.
}
Therefore,
\be\label{eq:I3result}
I_3\ll\log 2qT_2.
\ee

Combining \eqref{eq:I1result}, \eqref{eq:I2result}, \eqref{eq:I4result},
and \eqref{eq:I3result}, the theorem follows.

\section{Proof of Theorem~\ref{thm:main1}}
\label{proof main 1}

We prove only the first result. The extension of the
theorem for cubefree moduli is proved using a similar
argument.

In what follows, all constants $c_j$ are positive and
depend only on $\theta$. Put
$\theta'\defeq\frac12(\theta+\frac13)$,
so $\theta'$ depends only on $\theta$,
and $\theta>\theta'>\frac13$.

Applying Lemma~\ref{lem:Burgess} to the nonprincipal character
$\chi\defeq\chi_1\overline\chi_2$, it follows that
there are constants $c_3,c_4$ and a prime $p_0\ge 3$ such that
\be\label{eq:claim2}
p_0\le c_3 q^{\theta'}
\mand
|\chi_1(p_0)-\chi_2(p_0)|\ge\frac{c_4}{(\log q)^2}.
\ee

Let $T\ge c_1 q^\theta$, where
$c_1\ge 1$; note that $\log qT\ll\log T$.
Let $\Delta$ be any real number such that $0<\Delta<T$.
For $j=1$ or $2$, let $\cS_j(T)$ be the multiset of
zeros $\rho=\beta+i\gamma$ of $L(s,\chi_j)$ with $T<\gamma<T+\Delta$,
where the number of times each zero occurs in $\cS_j(T)$ is the same as
its multiplicity. Applying Theorem~\ref{thm:main2} with
$x\defeq p_0$, $T_1\defeq T$, and $T_2\defeq T+\Delta$, we have
\be\label{eq:junk}
\ssum{\rho\in\cS_j(T)}p_0^\rho
=-\frac{\Delta\log p_0}{2\pi}\chi_j(p_0)+O(p_0\log T\,\log\log 2p_0)
\qquad(j=1\text{~or~}2)
\ee
(here we have used the fact that $\langle p_0\rangle\ge 1$ and that $\log qT\ll\log T$).

Now suppose that $\cS_1(T)$ and $\cS_2(T)$ are equal (as multisets).
Taking the difference of the estimates in \eqref{eq:junk}, we get that
$$
\frac{\Delta\log p_0}{2\pi}\big(\chi_1(p_0)-\chi_2(p_0)\big)
\ll p_0\log T\,\log\log 2p_0.
$$
Using \eqref{eq:claim2} and the bound
$q^{\theta'}(\log q)^3\ll_\theta q^\theta$, we derive that
$$
\Delta\mathop{\,\ll\,}\limits_\theta q^{\theta'}
(\log q)^2\log T\,\frac{\log\log 2p_0}{\log p_0}
\mathop{\,\ll\,}\limits_\theta \frac{q^\theta\log T}{\log q}.
$$
But this is plainly impossible if we define
$$
\Delta\defeq \frac{c_2 q^\theta\log T}{\log q}
$$
for some suitably large choice of $c_2$. With this choice of $\Delta$,
the contradiction implies that either $\Delta\ge T$,
 or else $\cS_1(T)\ne\cS_2(T)$.
However, for a sufficiently large constant $c_1$, we have
$$
\frac{c_2 q^\theta\log T}{\log q}
<T\qquad(T\ge c_1 q^\theta).
$$
In other words, $\Delta<T$, and thus $\cS_1(T)\ne\cS_2(T)$.
The theorem is proved.

\bigskip\subsection{Acknowledgments}
The author thanks Steve Gonek,
Athanasios Sourmelidis, and Alexandru Zaharescu for helpful 
comments and for pointing out some relevant literature.
The author also thanks the referees for numerous valuable comments, which
led to substantial improvements.
For example, Theorem~\ref{thm:main1} was first
established by the author only for $\theta>2$. The suggestion of using
Lemma~\ref{lem:Burgess} to sharpen that theorem came from a referee, and
this contribution is gratefully acknowledged.

\end{document}